\documentclass{amsart}
\usepackage[latin1]{inputenc}
\usepackage[all]{xy}
\usepackage{amssymb, amsmath, amscd, geometry}
\usepackage{graphicx}
\usepackage{setspace}
\usepackage{ulem}

\numberwithin{equation}{section}
\input xy
\xyoption{all}

\usepackage{latexsym}
\usepackage[usenames]{color}
\usepackage{epic} 
\usepackage{mathrsfs}
\usepackage{euscript}
\usepackage{rotating}
\usepackage{verbatim}
\usepackage{tikz-cd}

\newtheorem{lemma}{Lemma}[section]
\newtheorem{theorem}[lemma]{Theorem}
\newtheorem{corollary}[lemma]{Corollary}

\newtheorem{prop}[lemma]{Proposition}

\newtheorem*{theorem*}{Theorem}

\theoremstyle{definition}
\newtheorem{definition}{Definition}[section]
\newtheorem{remark}{Remark}[section]

      \newcommand{\R}{{\mathbb R}}
      \newcommand{\C}{{\mathbb C}}

\newcommand{\iny}{\hookrightarrow} 

\newcommand{\Tk}{S_1^{k,\mathrm{sa}}}

\newcommand{\uno}{1\!\!1}

\title[On cb and $(1,cb)$-summing maps with applications to quantum XOR games]{On the relation between completely bounded and $(1,cb)$-summing maps with applications to quantum XOR games}

\author{Marius Junge}
\address{Department of Mathematics\\University of Illinois at Urbana-Champaign\\1409 W. Green St. Urbana, IL 61891. USA}
\email{junge@math.uiuc.edu}

\author{Aleksander M. Kubicki}
\address{Departamento de An\'alisis Matem\'atico y Matem\'atica Aplicada \\
Facultad de Ciencias Matem\'aticas \\ Universidad Complutense de Madrid \\
Madrid 28040. Spain}
\email{akubicki@ucm.es}

\author{Carlos Palazuelos}
\address{Instituto de Ciencias Matem\'aticas (ICMAT)\\Departamento de An\'alisis Matem\'atico y Matem\'atica Aplicada \\
Facultad de Ciencias Matem\'aticas \\ Universidad Complutense de Madrid \\
Madrid 28040. Spain}
\email{cpalazue@mat.ucm.es}

\author{Ignacio Villanueva} 
\address{Instituto de Ciencias Matem\'aticas (ICMAT)\\Departamento de An\'alisis Matem\'atico y Matem\'atica Aplicada \\
Facultad de Ciencias Matem\'aticas \\ Universidad Complutense de Madrid \\
Madrid 28040. Spain}
\email{ignaciov@mat.ucm.es}

\thanks{MJ is partially supported by the NSF grants DMS 1800872 and Raise-TAG183917. A. M. K is partially supported by the  European Research Council (ERC) (grant agreement No 648913) and Spanish MICINN project MTM2014-57838-C2-2-P. C.P. is partially supported by the MICINN project PID2020-113523GB-I00 and by QUITEMAD+-CM, P2018/TCS4342, funded by Comunidad de Madrid. I. V. is partially supported by the MICINN project PID2020-116398GB-I00. C.P. and I. V. are  partially supported by Grant CEX2019-000904-S funded by MCINN/AEI/ 10.13039/501100011033.}

\addtocounter{tocdepth}{-1}

\begin{document}


\maketitle

\begin{abstract}
	In this work we show that, given a linear map from a general operator space into the dual of a C$^*$-algebra, its completely bounded norm is upper bounded by a universal constant times its $(1,cb)$-summing norm. This problem is motivated by the study of quantum XOR games in the field of quantum information theory. In particular, our results imply that for such games entangled strategies cannot be arbitrarily better than those strategies using one-way classical communication.
\end{abstract}

\section{Introduction and main results}

During the last years there have been many interactions between quantum information and the fields of operator algebras and operator spaces/systems. To mention a few examples, free probability has been successfully applied in the study of quantum channel capacities \cite{CoNe10, CoNe11}, operator systems and operator algebras techniques have been recently used to study synchronous games \cite{HMPS19, LMPRSSTW11, PSSTW16}  and operators spaces have been key to solve several problems on nonlocal games and Bell inequalities \cite{PaVi16}. In fact, these connections also go in the converse direction, as it is shown by the new proofs of Grothendieck's Theorem for operator spaces based on the use of the Embezzlement state \cite{ReVi14}, the proof of new embeddings between noncommutative $L_p$-spaces \cite{JuPa16}  and certain operator algebras \cite{HGJ19} by using some classical protocols in quantum information and,  probably the most notable example, the recent resolution of the famous Connes Embedding Problem by using techniques from quantum computer sciences \cite{JNVWY}.

The main goal of this paper is to study the relation between certain norms defined on linear maps from a general operator space $X$ to the dual of a C$^*$-algebra $A^*$. This problem has a clear mathematical motivation, since some fundamental results such as the noncommutative versions of Grothendieck's Theorem, can be read in similar terms. However,  in the spirit of the previous paragraph, in the second part of the paper we explain that this problem, when restricted to the case where both $X$ and $A$ are matrix algebras, is equivalent to the study of certain values of the so-called quantum XOR games, hence stressing the close connection between pure mathematical problems and some questions motivated by quantum information theory.

In order to state our main results we need to introduce some elements. Let us recall that an operator space $X$ is a closed subspace of $B(H)$ \cite{EffrosRuanBook, Pisierbook}. For any such subspace the operator norm on $B(H)$ automatically induces a sequence of matrix norms $\|\cdot\|_d$ on $M_d(X)$, $d\geq 1$, via the inclusions  $M_d(X) \subseteq M_d(B(H))\simeq B(H^{\oplus d})$. In this way, given a linear map $T:X\rightarrow Y$  between two operator spaces $X$ and $Y$, we say that $T$ is completely bounded if $$\|T\|_{cb}:=\sup_d\|\uno\otimes T:M_d(X)\rightarrow M_d(Y)\|<\infty.$$

The study of operator spaces was initiated in \cite{Ruan88} and can be understood as a noncommutative version of Banach space theory. Since then, an important line of research has been devoted to developing the  ``Grothendieck's program''  for operator spaces (see for instance \cite{EfRu91, HaMu08, Junge05, PiSh02, Xu06}). A crucial definition in the local theory of Banach spaces is that of absolutely $p$-summing maps, as those linear maps between Banach spaces $T:X\rightarrow Y$ such that $$\pi_p(T):=\|\uno\otimes T:\ell_p\otimes_{\epsilon}X\rightarrow \ell_p(Y)\|<\infty.$$ Motivated by the great relevance of these maps \cite{DiJaTo95}, in \cite{Pisier98} Pisier introduced and studied a noncommutative analogue in the context of operator spaces. Given a linear map between operator spaces $T:X\rightarrow Y$, we say that $T$ is completely $p$-summing if  $$\pi_p^0(T):=\|\uno\otimes T:S_p\otimes_{min}X\rightarrow S_p(Y)\|<\infty.$$ Note that the previous definition requires of the highly nontrivial concept of noncommutative vector-valued $L_p$-spaces, which was also developed in \cite{Pisier98}.

However, the noncommutative context admits some other generalization of $p$-summing maps. Here, we will deal with the $(p,cb)$-summing maps, introduced by the first author in \cite{JungeHab} (see also \cite{ParJu16}) and which can be understood as an intermediate definition between the one for Banach spaces and the one for operator spaces  above. Given an operator space $X$ and a Banach space $Y$, a linear map $T:X\rightarrow Y$ is said to be $(p,cb)$-summing if 
\begin{align*}
\pi_{(p,\, cb)}(T):=\|\uno\otimes T:\ell_p\otimes_{min}X\rightarrow \ell_p(Y)\|<\infty.
\end{align*}

It is clear from the previous definitions that for every linear map $T:X\rightarrow Y$ between operator spaces, the inequality  $\max\{\|T\|_{cb},$ $ \pi_{(p,\, cb)}(T)\}\leq \pi_p^0(T)$ holds.  However, there is no general relation between the quantities $\|T\|_{cb}$ and $\pi_{(p,\, cb)}(T)$. That is, one can find examples of linear maps and operator spaces for which $\|T\|_{cb}<\infty$ and $\pi_{(p,\, cb)}(T)=\infty$ and also for which $\|T\|_{cb}=\infty$ and $\pi_{(p,\, cb)}(T)<\infty$. 

In this work we study the relation between $\|T\|_{cb}$ and $\pi_{(1,\, cb)}(T)$ for maps $T$ defined from a general operator space $X$ to the dual of a C$^*$-algebra $A^*$. Our main result is as follows.
\begin{theorem}\label{TMain}
	There exists a universal constant $K$ such that for any linear map $T:X\rightarrow A^*$, where $X$ is an operator space and $A$ is a C$^*$-algebra, we have $$\|T\|_{cb}\leq K\pi_{1,cb}(T).$$
\end{theorem}

In order to prove Theorem \ref{TMain} we will need to study the quantity $\Gamma_{R\cap C}$ (the factorizable ``norm'' through the operator space $R\cap C$) and prove that it fits very well in our context. Once this is done, Theorem \ref{TMain} will follow as an application of the noncommutative Little Grothendieck's Theorem. In fact, we will prove a stronger result than the one stated above, namely $\Gamma_{R\cap C}(T)\leq K\pi_{1,cb}(T)$. It is worth mentioning that one cannot expect to have a converse inequality in Theorem \ref{TMain}; not even in the commutative case. That is, there exist maps $T:\ell_\infty\rightarrow \ell_\infty^*$ for which $\|T\|_{cb}<\infty$ and $\pi_{(1,\, cb)}(T)=\infty$ (see Section \ref{Sec:QXOR} for details). Let us also mention that we do not know if the constant $K$ in Theorem \ref{TMain} can be taken equal to one.  However, we stress that the techniques used in the present work lead irremediably to $K > 1$.

Theorem  \ref{TMain} can be read in the context of quantum XOR games \cite{ReVi15} when $X=A=M_n$ is the C$^*$-algebra of $n\times n$ complex matrices. Quantum XOR games are collaborative games where a referee asks some (quantum) questions to a couple of players, usually called Alice and Bob, who must answer with outputs $a,\, b \in \{\pm1\}$. According to the questions and the parity of the answers, $ab$, the players win or loose the game. It turns out that these games can be identified with selfadjoint matrices $G\in M_{nm}$ such that $\|G\|_{S_1^{nm}}\leq 1$, where $S_1^{nm}$ denotes the corresponding $1$-Schatten class (see Section \ref{Sec:QXOR} for details). Moreover, the largest possible winning probability of the game depending on the type of strategies performed by the players can be expressed by means of norms on $\hat{G}:M_n\rightarrow M_m$, where $\hat{G}$ is the linear map associated to the matrix $G\in M_{nm}$ according to the algebraic identification $M_{nm}=L(M_n\rightarrow M_m)$. In this context, if we denote by $\beta^*(G)$ the largest bias\footnote{For some reasons that will become clear in Section \ref{Sec:QXOR}, when working with XOR games one usually works with the bias $\beta=2P_{win}- 1$ rather than with the winning probability $P_{win}$.} of the game $G$ when the players are allowed to perform entangled strategies, it is known that $\beta^*(G)=\|\hat{G}:M_n\rightarrow S_1^m\|_{cb}$ (see Section \ref{Sec:QXOR}). On the other hand, if we denote by $\beta_{owc}(G)$ the largest bias of the game when the players are allowed to send one-way classical communication as part of their strategies, we will show in Section  \ref{Sec:QXOR} that $\beta_{owc}(G)\approx\pi_{1,cb}(\hat{G}:M_n\rightarrow S_1^m)$, where  $\approx$ means equivalence up to a universal constant. Hence, Theorem \ref{TMain} above leads to the following consequence.
\begin{corollary}\label{Cormain}
	Let $G$ be a quantum XOR games. Then, $$\beta^*(G)\leq K'\beta_{owc}(G)$$for a certain universal constant $K'$.
\end{corollary}

One of the main goals of quantum information theory is to find scenarios where quantum entanglement is ``much more powerful'' than classical resources. When working with classical XOR games \cite{Cleve04} (see also \cite{PaVi16}), it is very easy to see that the one-way communication of classical information is as powerful as possible. That is, those games can always be won with probability one if the players can use classical communication as part of their strategy. On the contrary, as a consequence of the classical Grothendieck theorem, entanglement is a quite limited resource to play classical XOR games, providing only small advantages over classical strategies. The situation changes dramatically for \textit{quantum} XOR games. Within this more general family of games there exist instances for which the use of entanglement allows to attain biases which are unboundedly larger than the best achievable bias by players sharing only classical randomness \cite{ReVi15}. On the other hand, since the questions now are quantum states, classical communication is not enough to win with certainty. In fact, there exist quantum XOR games for which sharing one-way classical communication does not provide any advantage at all (see Section  \ref{Sec:QXOR} for further clarification on the previous statements). This new phenomenology motivates us to ask whether there exist quantum XOR games for which quantum entanglement allows Alice and Bob to answer much more successfully than using one-way classical communication. Corollary \ref{Cormain} says that this is not the case. Hence, one needs to consider more involved tasks than winning quantum XOR games in order to find examples for which quantum entanglement is much better than sending classical information.

The structure of the paper is the following. In Section \ref{Sec:Preliminaries} we introduce some notation and basic results that will be used along the paper. Section \ref{Sec: Main result} will be devoted to proving Theorem \ref{TMain}. Finally, in Section \ref{Sec:QXOR} we will introduce quantum XOR games and we will explain how different values of these games can be written in tems of norms on linear maps between some operator spaces. As a consequence of this mathematical formulation for the different values of quantum XOR games we will see how Corollary \ref{Cormain} can be obtained from Theorem \ref{TMain}.

\section{Preliminaries and some basic results}\label{Sec:Preliminaries}

In this section we introduce some tools and well known results that we will use later. We will add the proof for some statements which cannot be found explicitly in the literature.  We assume the reader to be familiar with the basic elements of Banach spaces \cite{Tomczak} and operator spaces \cite{Pisierbook}.

\subsection{Absolutely $p$-summing maps and completely $p$-summing maps}\label{subsec p-summing}

Given a linear map $T:X\rightarrow Y$ between two Banach spaces and $1\leq p<\infty$, we say that $T$ is \textit{absolutely $p$-summing} if
\begin{align}\label{Equation p-summing}
\pi_p(T):=\|id\otimes T:\ell_p\otimes_\epsilon X\rightarrow \ell_p(X)\|<\infty,
\end{align}where here $\ell_p\otimes_\epsilon X$ denotes the (complete) injective tensor product and $\ell_p(X)$ is the corresponding vector valued $L_p$-space. It is not difficult to see that $\pi_p$ is a norm on the set of all absolutely $p$-summing maps. The \textit{factorization theorem} for these maps \cite[Theorem 2.13]{DiJaTo95} states that $T:X\rightarrow Y$  is absolutely $p$-summing if and only if there exist a regular Borel probability measure $\mu$ on the unit ball of the dual space of $X^*$, $B_{X^*}$, a closed subspace $E_p\subseteq L_p(C(B_{X^*}), \mu)$ and a linear map $u:E_p\rightarrow Y$ with $\|u\|=\pi_p(T)$ such that the following diagram commutes:

$$\xymatrix@R=0.6cm@C=1.5cm {{C(B_{X^*})}\ar[r]^{i} &
	{L_p(\mu)} \\{j(X)}\ar@{}[u]|-*[@]{\subset}\ar[r]^{i|_{j(X)}}
	& {E_p}\ar@{}[u]|-*[@]{\subset}\ar[d]^{u}\\{X}\ar[u]^{j}\ar[r]^{T} & {Y}}$$
Here, $j:X\iny C(B_{X^*})$ is the canonical embedding and $i:C(B_{X^*})\rightarrow L_p(C(B_{X^*}), \mu)$ is the identity map. 

The case $p=2$ deserves some attention in this work. First, note that in this case the space $E_2$ can be replaced by the whole space $L_2(C(B_{X^*}), \mu)$  (by complementation) and there is no need to consider the restriction of $i$ to $j(X)$ \cite[Corollary 2.16]{DiJaTo95}. Moreover, it is not difficult to see from this factorization theorem that $2$-summing operators have the \textit{extension property} \cite[Theorem 4.15]{DiJaTo95}: Given a $2$-summing operator $T:X\rightarrow Y$ and any isometry $j:X\iny \tilde{X}$, there exists an extension $\tilde{T}:\tilde{X}\rightarrow Y$ (so that $T=\tilde{T}\circ j$) verifying $\pi_2(T)=\pi_2(\tilde{T})$.

Motivated by the great relevance of absolutely $p$-summing maps in the local theory of Banach spaces, in \cite{Pisier98} Pisier developed the theory of completely $p$-summing maps in the context of operator spaces.  Given a linear map $T:X\rightarrow Y$ between two operator spaces and $1\leq p<\infty$, we say that $T$ is \textit{completely $p$-summing} if $$\pi_p^0(T):=\|id\otimes T:S_p\otimes_{min} X\rightarrow S_p(X)\|<\infty,$$where here $S_p\otimes_{min} X$ denotes the minimal tensor product in the category of operator spaces and $S_p(X)$ is the corresponding non-commutative vector valued $L_p$-space. 

It is interesting to note that completely $p$-summing maps verify a factorization theorem analogous to the one for absolutely $p$-summing maps. However, in order to explain that result we need to recall some definitions about ultraproducts of Banach spaces and operator spaces. We refer to \cite{Hein80} for a detailed exposition on ultraproducts of Banach spaces and to \cite[Section 2.8]{Pisierbook} for the operator space case. Given a family of Banach spaces $(X_i)_{i\in I}$ and a nontrivial ultrafilter $\mathcal U$ on the set $I$, denote by $\ell$ the set of elements $(x_i)_{i\in I}$ with $x_i\in X_i$ for every $i$ and such that $\sup_i\|x_i\|<\infty$. We equip this space with the norm $\|x\|=\sup_i\|x_i\|$. Let us now denote by $\nu_\mathcal U$ the subspace of $\ell$ given by the elements $x$ such that $\lim_{\mathcal U}\|x_i\|=0$. The quotient $\ell/\nu_\mathcal U$ is a Banach space called ultraproduct of the family $(X_i)_{i\in I}$ and denoted by $\prod X_i/\mathcal U$. Note that if $[x]$ is the equivalence class associated to an element $(x_i)_i$, then $\|[x]\|=\lim_\mathcal U\|x_i\|$. If in addition $X_i$ is endowed with an operator space structure for every $i$, we can endow the space $\prod X_i/\mathcal U$ with a natural operator space structure by defining $M_n(\prod X_i/\mathcal U)=\prod M_n(X_i)/\mathcal U$ for every $n\in \mathbb{N}$. It can be seen that given a family of completely bounded maps $T_i:X_i\rightarrow Y_i$ for every $i$, one can define a linear map $\hat{T}:\prod X_i/\mathcal U\rightarrow \prod Y_i/\mathcal U$ by $\hat{T}([(x_i)_i])=[(T(x_i))_i]$ which verifies $\|T\|_{cb}\leq \sup_i\|T_i\|_{cb}$. It is also interesting to mention that ultraproducts respect isometries and quotients both in the Banach space category and in the operator space category.

The \textit{factorization theorem} for completely $p$-summing maps \cite[Remark 5.7]{Pisier98} states that given a linear map $T:X\rightarrow Y$ between operator spaces, such that $X\subset B(H)$, there exist an ultrafilter $\mathcal U$ over an index set $I$,  families $(a_i)_i$, $(b_i)_i$ in the unit sphere of $S_{2p}(H)$, a closed (operator) space $E_p\subseteq \prod S_p/\mathcal U$ and a linear map $u:E_p\rightarrow Y$ with $\|u\|_{cb}=\pi^0_p (T)$ such that the following diagram commutes:
$$\xymatrix@R=0.6cm@C=1.5cm {{\prod B(H)/\mathcal U}\ar[r]^{\mathcal M} &
	{\prod S_p/\mathcal U} \\{j(X)}\ar@{}[u]|-*[@]{\subset}\ar[r]^{\mathcal M|_{j(X)}}
	& {E_p}\ar@{}[u]|-*[@]{\subset}\ar[d]^{u}\\{X}\ar[u]^{j}\ar[r]^{T} & {Y}}$$
Here, $j:X\iny \prod B(H)/\mathcal U$ is the complete isometry defined as $j(x)=[(x)_{i\in I}]$ and $\mathcal M:\prod B(H)/\mathcal U\rightarrow \prod S_p/\mathcal U$ is the linear map defined by the family $(M_i)_i$, where $M_i:B(H)\rightarrow S_p(H)$ is defined as $M_i(x)=a_ixb_i$ for every $i\in I$. In the previous picture,  $E_p=\overline{\mathcal M(j(X))}$.

As mentioned in the Introduction, one can define an intermediate notion between absolutely $p$-summing and completely $p$-summing maps. Indeed, given an operator spaces $X$ and a Banach space $Y$, a linear map $T:X\rightarrow Y$ is said to be \textit{$(p,cb)$-summing} (see \cite{JungeHab}, \cite{ParJu16}) if 
\begin{align}\label{Equation (p,cb)-summing}
\pi_{(p,\, cb)}(T):=\|\uno\otimes T:\ell_p\otimes_{min}X\rightarrow \ell_p(Y)\|<\infty.
\end{align}
\begin{remark}\label{Rem-factorization (p,cb)}
It was observed by Pisier \cite[Remark 5.11]{Pisier98}, that $T:X\rightarrow Y$ is $(p,cb)$-summing if and only if it verifies a similar factorization theorem to the one for completely $p$-summing maps but where, in this case, $\|u\|=\pi_{(p,\, cb)}(T)$.
\end{remark}

Finally, in order to study the previous type of maps in the context of our main Theorem \ref{Main Thm}, we will need two more definitions related with summing properties of linear maps. In order to introduce them, we first need to recall some additional notions in operator space theory. Given a complex Hilbert space $H$, the operator space structures defined by the  isometric identifications 
\begin{align}\label{Def R-C}
H\simeq B(H,\C)\hspace{0.5 cm}\text{   and   }\hspace{0.5 cm}H\simeq B(\C,H)
\end{align} are the row and column operator space structures on $H$, denoted by $R_H$ and $C_H$, respectively. When the underlying Hilbert space is $H= \ell_2$, we use the simpler notation $R$ and $C$. Moreover, we can also define the $R_H\cap C_H$ operator space structure on $H$ by means of the embedding 
\begin{align}\label{embedding intersec}
j:R_H\cap C_H\rightarrow R_H\oplus_\infty C_H,
\end{align}defined as $j(x)=(x,x)$. Finally, the $R_H+ C_H$ operator space structure on $H$ can be defined so that $R_H+C_H=(R_H\cap C_H)^*$ completely isometrically. The following stability properties under ultraproducts  will play a role later on:

\begin{remark}\label{remark_ultraproducts}
	It is well known that, at the Banach space level, the ultraproduct of a family of Hilbert spaces $(H_i)_i$, $\hat{H}=\prod H_i/\mathcal U$, is a Hilbert space. Furthermore, according to the definition of $R_H$ and $C_H$ and the comments above, it is not difficult to see that  $$\prod R_{H_i}/\mathcal U=R_{\hat{H}}\hspace{0.5 cm}\text{   and   }\hspace{0.5 cm}\prod C_{H_i}/\mathcal U=C_{\hat{H}}.$$ Moreover, the properties of the ultraproducts together with the definition of $R_H\cap C_H$ via the embedding (\ref{embedding intersec}) ensures that $$\prod (R_{H_i}\cap C_{H_i})/\mathcal U=R_{\hat{H}}\cap C_{\hat{H}}.$$
	In fact, this stability under ultraproducts is a property of any homogeneous Hilbertian operator space (as, e.g.,  $R_H,\, C_H$ or $R_H \cap C_H$). See \cite[Lemma 3.1 and remarks in page 82]{PisierOH}.
\end{remark}

With the previous definitions at hand, we can come back to summing properties of linear maps. Given a mapping $T:X\rightarrow Y$ between an operator space $X$ and a Banach space $Y$, we say that $T$ is \textit{$(2,RC)$-summing} if 
\begin{equation}\label{Def:pi_2,RcapC}
\pi_{2,RC}(T):=\|id\otimes T:(R\cap C)\otimes_{min} X\rightarrow \ell_2(Y)\|<\infty,
\end{equation}
and we say that the map is \textit{$(2,R+C)$-summing} if 
\begin{equation}\label{Def:pi_2,R+C}
\pi_{2,R+C}(T):=\|id\otimes T:(R+ C)\otimes_{min} X\rightarrow \ell_2(Y)\|<\infty.
\end{equation}

Compare the previous two definitions with the definition of $(2,cb)$-summing map given in Equation (\ref{Equation (p,cb)-summing}) for $p=2$, where the operator space considered in $\ell_2$ is the so called $OH$.

\subsection{Little Grothendieck's theorem}

Although most maps between Banach spaces are not $p$-summing for any $p$, a famous result, called \textit{little Grothendieck's theorem}, asserts that every linear map $T:C(K)\rightarrow L_2(\mu)$, where $K$ is a compact space and $\mu$ is any measure verifies that $\pi_2(T)\leq K_{LG}\|T\|$. Here, $K_{LG}=\sqrt{\pi/2}$ in the real case and $K_{LG}=2/\sqrt{\pi}$ in the complex case.

There is also a noncommutative version of this result, which was first proved in \cite{Pisier78} and later in \cite{Haagerup85} (with an improvement in the constant). This result is usually referred to as \textit{non-commutative little Grothendieck's theorem}\footnote{In order to see that this result generalizes the classical little Grothendieck's theorem one must use the \textit{domination theorem} for $2$-summing maps \cite[Theorem 2.13]{DiJaTo95}. We omit this result here because we will not use it.}.

\begin{theorem}\label{NClittleGrothendieck}
	Let $A$ be a C$^*$-algebra and $H$ be a Hilbert space. Then, for any bounded linear map $T:A\rightarrow H$ there exist states $f_1$ and $f_2$ on $A$ such that $$\|T(x)\|\leq \|T\|\Big(f_1(x^*x)+ f_2(xx^*)\Big)^\frac{1}{2}$$for every $x\in A$.
\end{theorem}

Although the following corollary is folklore, we have not found any proof in the literature. Since it will be crucial for us, we give some hints about its proof.
\begin{corollary}\label{key cor}
	Let $A$ be a C$^*$-algebra and $H$ be a Hilbert space. Then, any bounded linear map $T:A\rightarrow H$ verifies that $$\|T:A\rightarrow R_H+C_H\|_{cb}\leq 2\|T\|.$$
\end{corollary}
\begin{proof}
The key idea is to use the non-commutative little Grothendieck theorem to decompose $T$ into two maps $T = T_1 + T_2 $ such that
	\begin{equation}\label{key cor_Eq1}
	\max\{\| T_1: A \rightarrow R_H \|_{cb}, \, \| T_2 : A \rightarrow C_H \|_{cb}\} \le \|  T : A \rightarrow H \|.
	\end{equation}
	With this at hand, the statement is obtained noticing that:
	\begin{align*}
	\| T: A \rightarrow R_H + C_H \|_{cb} \le \| T_1 \|_{cb} + \| T_2 \|_{cb} \le 2 \| T \|.
	\end{align*} 
	
	Therefore, the main part of the proof consists on constructing $T_1$ and $T_2$ with the claimed properties. For that,  observe that the states $f_1$ and $f_2$ from Theorem \ref{NClittleGrothendieck}  define pre-inner products $\langle  x , y \rangle_1  : = f_1(x y^*), \ \langle x, y \rangle_2  : = f_2(x^* y),$ for any $x, \, y \in A$, which naturally induce Hilbert spaces $H_1$, $H_2$ in the obvious way.
	
	Given that, we consider the Hilbert space $H_1 \oplus H_2 $ and the injections:
	$$
	\begin{array}{r c  c c r c c c}
	j_1: & A  &\longrightarrow & H_1 \oplus H_2    & ,\qquad \qquad j_2  :& A  &\longrightarrow & H_1 \oplus H_2   \\
	&  x  & \mapsto  & [x] \oplus 0    & &   x  & \mapsto  & 0 \oplus [x]
	\end{array} .
	$$
	Next, we are interested in the projection, $p$, on the subspace $E = \overline{ \{   [x] \oplus [x] \, : \, x \in A\} } \subset H_1 \oplus H_2 $, in which we can understand the original map $T$ acting as:
	$$
	\begin{array}{r c  c c }
	\tilde T: & E  &\longrightarrow & H      \\
	&  [x]\oplus [x]  & \mapsto  & T(x)  
	\end{array} .
	$$
	
	The maps we are looking for are  constructed composing the previous building blocks:
	$$
	\text{ for $i=1 $ or $2$,} \qquad T_i:  A   \stackrel{j_i}{\longrightarrow}  H_1\oplus H_2    \stackrel{p}{\longrightarrow} E  \stackrel{\tilde T}{\longrightarrow}  H .   
	$$
	
	One can check that the previous maps are well defined, that they verify $T=T_1 + T_2$ by construction and that \eqref{key cor_Eq1} is implied by the claim in Theorem \ref{NClittleGrothendieck}.
	
\end{proof}

The following is a consequence of the previous corollary combined with the extension property of $(2,RC)$-summing maps.
\begin{corollary}\label{cor 2}
	Let $X$ be an operator space and $H$ be a Hilbert space. Then, any $(2,RC)$-summing map $T:X\rightarrow H$ verifies that $$\|T:X\rightarrow R_H+C_H\|_{cb}\leq 2\pi_{2,RC}(T).$$
\end{corollary}
\begin{proof}
	Let $K$ be a Hilbert space such that $j:X\iny B(K)$ is a complete isometry. According to \cite[Proposition 0.4]{JuPi95}, there exists a linear map $\tilde{T}:B(K)\rightarrow H$ such that $T= \tilde{T}\circ j$ and $\pi_{2,RC}(\tilde{T})\leq \pi_{2,RC}(T)$. Now, it follows from Corollary \ref{key cor} that $$\|\tilde{T}:B(K)\rightarrow R_H+C_H\|_{cb}\leq 2\|\tilde{T}:B(K)\rightarrow H\|\leq 2\pi_{2,RC}(\tilde{T})\leq 2\pi_{2,RC}(T).$$ Hence, since $j:X\iny B(K)$ is a complete isometry, the previous inequality implies that 
	$$\|T:X\rightarrow R_H+C_H\|_{cb}\leq 2\pi_{2,RC}(T)$$as we wanted. 
\end{proof}

Finally, we will also need  the following (well known) lemma, which relates the completely bounded norm and the 2-summing norm.
\begin{lemma}\label{lemma: 2-sum}
	Let $H$ be a Hilbert space endowed with one of the following operator space structures: $R$, $C$ or $R\cap C$. Then, for any linear map $T:C(K)\rightarrow H$ we have $\|T\|_{cb}=\pi_2(T)$, where $C(K)$ denotes the space of complex continuous function on a compact space $K$.
\end{lemma}
\begin{proof}
	Let us first show the result for $R$. To this end, we use \cite[Proposition 5.11]{PisierOH} to state that $$\|T:C(K)\rightarrow R_H\|_{cb}=\|id\otimes T:R\otimes_{min}C(K)\rightarrow R\otimes_{min}R_H\|.$$Now, the fact that $C(K)$ is a commutative C$^*$-algebra guarantees that $R\otimes_{min}C(K)=\ell_2\otimes_\epsilon C(K)$ isometrically (see for instance \cite[Proposition 1.10]{Pisierbook}). Moreover, it is easy to see that, also isometrically, $R\otimes_{min}R_H=\ell_2(H)$. Hence, 
	$$\|T:C(K)\rightarrow R_H\|_{cb}=\|id\otimes T:\ell_2\otimes_\epsilon C(K)\rightarrow \ell_2(H)\|=\pi_2(T).$$
	
	The proof for the $C$ structure is completely analogous and the proof for $R\cap C$ follows easily from its definition  and the estimates for $R$ and $C$.
\end{proof}

\subsection{Weights on Banach spaces}\label{subsec:Weights}

Given a Banach space $X$, following \cite{PisierOH} we denote
$$
(X \otimes \overline X)_+  = \lbrace	u \in X \otimes \overline{X} \ : \ \langle u, \xi \otimes \overline \xi \rangle \ge 0 \quad \forall \xi \in X^*	 \rbrace.
$$
Here $\overline{X}$ is the Banach space conjugate to $X$, that is simply $X$ itself but equipped with the complex conjugate multiplication. Note that $(\overline{X})^*$ can be naturally identified with  $\overline{X^*}$. Moreover, the elements in $(X \otimes \overline X)_+$ can be understood as positive sesquilinear forms on $X^* \times X^*$ and can be always written as:
$$
u = \sum_{i=1}^n x_i \otimes \overline{x_i},
$$
for some finite set $x_1, \cdots, x_n \in X$.

Since $(X \otimes \overline X)_+$ is a cone, it naturally defines an order in $(X \otimes \overline X)$. In particular, note that $\sum_{i=1}^n x_i \otimes \overline{x_i}\leq  \sum_{j=1}^m y_i \otimes \overline{y_i}$ if and only if $$\sum_{i=1}^n|\xi(x_i)|^2\leq  \sum_{j=1}^m |\xi(y_i) |^2\hspace{0.3 cm} \text{ for every $\xi\in X^*$.}$$

The following proposition, proved in \cite[Proposition 2.2]{Pisier_Factorization}, will be very useful later.
\begin{prop}\label{Prop Order}
Given two elements $u=\sum_{i=1}^n x_i \otimes \overline{x_i},\, v= \sum_{j=1}^m y_i \otimes \overline{y_i}$ in $(X \otimes \overline X)_+$, we have that  $u\leq v$ if and only if there is a contraction $a:\ell_2^m\rightarrow \ell_2^n$ such that $$(a\otimes id)\Big(\sum_{j=1}^me_i\otimes y_j\Big)=\sum_{i=1}^n e_i\otimes x_i,$$where here $(e_i)_i$ denotes any orthonormal basis and $id:X\rightarrow X$.
\end{prop}

\begin{definition}\label{Def_weight}
	We say that $ w : (X \otimes \overline{X} )_+  \rightarrow \mathbb{R}_+ $ is a weight on $(X \otimes \overline{X} )_+$ if for all $ u, \, v \in ( X \otimes \overline{X} )_+ $:
	\begin{enumerate}
		\item [i.] (positive homogeneity) $w (t u) = t\,  w(u)$ for any $t \geq 0$;
		\item [ii.]  (subadditivity)  $w (u+ v ) \le  w(u) + w(v)$;
		\item [iii.]  (monotonicity) if $ u \le v $, $w(u) \le w(v)$.
	\end{enumerate}
\end{definition} 

The appearance of this notion of weights in our work is in part due to the nice duality theory displayed by  the  \textit{gamma-norms} introduced by Pisier in \cite{Pisier_Factorization}. In particular, Theorem 6.1 in \cite{PisierOH}  will play an important role for us. In order to state it, we need to introduce the following generalization of 2-summing norms:
Given Banach spaces $X$, $Y$ and a weight $\omega$ on $(X \otimes \overline{X} )_+$, a linear map $u : X \rightarrow Y$ is said to be  $(2,w)$-summing if there exists a constant $C$ such that for any finite sequence of elements $(x_i)_i$  in $X$ we have
	$$
	\left( \sum_i \| u(x_i) \|^2_X \right)^{\frac{1}{2}} \le C  \left( w \big(\sum_i x_i \otimes \overline x_i \big) \right)^{\frac{1}{2}}. 
	$$ The infimum of the constants for which the previous holds will be denoted  $\pi_{2,w}(u)$. It is not difficult to show that $\pi_{2,w}$ is in fact a norm.

\begin{remark}
In fact, the norm $\pi_{2,RC}$ defined in Equation \eqref{Def:pi_2,RcapC} can be understood as the $\pi_{2,w}$ norm associated to the weight $w \left(  \sum_i x_i \otimes \overline{x}_i \right)  =  \big\| \sum_i e_i \otimes x_i   \big \|_{ (R\cap C) \otimes_{min} X }^2 $  on $(X \otimes \overline{X} )_+$. On the contrary, this is not the case for  $\pi_{2,R + C}$ in \eqref{Def:pi_2,R+C}. However, in Section \ref{Sec: Main result} we construct a related weight that circumvents this problem at the expense of a multiplicative factor of $2$ (see Lemma \ref{Lemma_weight_R+C} for an explicit statement).  
\end{remark}

Now we can state the duality theorem for gamma-norms:
\begin{theorem}[\cite{PisierOH}, Thm. 6.1]\label{Thm_gammas}
Given Banach spaces $X,$ $Y$ and weights $w_1$ on $(X \otimes \overline{X} )_+$ and $w_2$ on $(Y \otimes \overline{Y} )_+$, consider the semi-norm on $X \otimes Y$:
	$$
	\gamma(u) := \inf_{u = \sum_i x_i \otimes y_i}  w_1\left(\sum_i x_i \otimes \overline{x}_i  \right)^{1/2} \, w_2\left(\sum_i y_i \otimes \overline{y}_i  \right)^{1/2}.
	$$
	Given a linear form $V$ on $ X\otimes Y$, we define
	$$
	\gamma^*(V) = \sup_{u\in X\otimes Y \, : \, \gamma(u) \le 1} | V(u) |.
	$$
	
	Then, if $\gamma^*(V) < \infty$,
	$$
	\gamma^*(V) := \inf \{ \pi_{2,w_1}(v_1) \, \pi_{2,w_2}(v_2^*) \},
	$$
	where the infimum runs over all $v_1 : X \rightarrow H$, $v_2: Y \rightarrow H^*$ such the operator $v: X \rightarrow Y^*$ associated to $V$ factorizes as $v = v_2^*\circ v_1$. 
\end{theorem}
\color{black}

\section{Main result}\label{Sec: Main result}

In this section we will prove our main result, that we state again for convenience.
\begin{theorem}\label{Main Thm}
There exists a universal constant $K$ such that for any linear map $T:X\rightarrow A^*$, where $X$ is an operator space and $A$ is a C$^*$-algebra, we have $$\|T\|_{cb}\leq K\pi_{1,cb}(T).$$
\end{theorem}

Before proving the result, let us make some comments.

It follows from the proof of Theorem \ref{Main Thm} that the constant $K$ can be taken equal $8\sqrt{2}$. We did not attempt any optimization in terms of this constant. However, our proof inevitably leads to a constant strictly larger than one. Whether one can get $K=1$ in the previous statement seems an interesting problem (See Section \ref{Sec:QXOR} for a related problem in quantum information).

The fact that the image of $T$ is  in the dual of a C$^*$-algebra is crucial in Theorem \ref{Main Thm}, since one can find examples of (operator) spaces $X$, $Y$ for which $\|T:X\rightarrow Y\|_{cb}$ can be arbitrary larger than $\pi_{1,cb}(T:X\rightarrow Y)$. Indeed, this can be shown, for instance, by considering $X=CL_n$, the operator spaces associate to Clifford algebras  with $n$ generators \cite[Section 9.3]{Pisierbook}, and $Y=\max(\ell_2^n)$. With this choice, we have $\pi_{1,cb}(id:CL_n\rightarrow \ell_2^n)\leq 2$  \cite[Proposition 4.3.2]{JungeHab} and $\|id:CL_n\rightarrow \max(\ell_2^n)\|_{cb}\geq \sqrt{n}$ \cite[Theorem 10.4]{Pisierbook}.

Finally, recall that while it is known \cite[Corollary 5.5]{Pisier98} that  $$\pi_1^0(T)=\|id\otimes T:S_1\otimes_{min} X\rightarrow S_1(Y)\|=\|id\otimes T:S_1\otimes_{min} X\rightarrow S_1(Y)\|_{cb},$$$\pi_{1,cb}(T)=\|id\otimes T:\ell_1\otimes_{min} X\rightarrow \ell_1(Y)\|$ does not coincide in general  with $\|id\otimes T:\ell_1\otimes_{min} X\rightarrow \ell_1(Y)\|_{cb}$. Indeed, it follows from \cite{JuPa16} that $\|id\otimes T:\ell_1\otimes_{min} X\rightarrow \ell_1(X)\|_{cb}=\pi_1^0(T)$ and there are known examples showing that $\pi_1^0(T)$ can be much larger than $\pi_{1,cb}(T)$ for maps $T:B(H)\rightarrow S_1(H)$.

In order to prove Theorem \ref{Main Thm} we will need to introduce a new weight. To this end, let us first define, for any element $u\in \ell_2\otimes X$, the quantity 
\begin{align*}
\| u\|_{(R +_2 C) \otimes_{min} X} := \inf \left( \| T \|_{R \otimes_{min} X }^2  + \| S \|_  {C \otimes_{min} X}^2 \right)^{1/2},
\end{align*}where the infimum is taken  over $T,\,  S \in \ell_2 \otimes X $  such that $u = T + S$. Now, due to the homogeneity of $R$ and $C$, it is very easy to see that for any bounded operator $a:\ell_2\rightarrow \ell_2$, we have
\begin{align}\label{homog R+_2 C}
\|(a\otimes id)(u)\|_{(R +_2 C) \otimes_{min} X}\leq \|a\|\| u\|_{(R +_2 C) \otimes_{min} X}.
\end{align}

Indeed, given such $u$ and $a$, we have
\begin{align*}
\|(a\otimes id)(u)\|_{(R +_2 C) \otimes_{min} X} \le  \inf_{T,\, S \, : \, u = T+S}   \left( \| (a \otimes \mathrm{id}) (T) \|_{ R \otimes_{min} X }^2   +   \| (a \otimes \mathrm{id}) ( S )\|_{ C \otimes_{min} X }^2 \right)^{1/2},
\end{align*}
where the inequality follows from the fact that $(a \otimes \mathrm{id}) (T) +(a \otimes \mathrm{id} )(S) = (a \otimes \mathrm{id} )(T+S) = (a \otimes \mathrm{id}) (u) $. Then, using that $R$ and $C$ are homogeneous operator spaces, the completely bounded norm of $a$ coincides with its norm when viewed as an operator $a:R \rightarrow R $ and $a:C \rightarrow C$. Hence, in the previous expression, $ \|(a \otimes \mathrm{id})(T) \|_{R \otimes_{min} X} \le \|a\|\| T \|_{ R \otimes_{min}X}$ and $\|(a \otimes \mathrm{id})(S)\|_{C \otimes_{min} X }  \le \|a\| \|S\|_{ C \otimes_{min} X}$. This straightforwardly implies \eqref{homog R+_2 C}.

Although the following lemma was essentially proved in \cite[Lemma 4.2.1]{JungeHab} (for a different  definition of weight) we add it here for completeness.
\begin{lemma}\label{Lemma_weight_R+C}
	For any operator space $X$, there exists a weight $w$ on $(X \otimes \overline{X} )_+$ such that for any $x_1, \ldots, x_n \in X$,
\begin{align}\label{inequality weight}
\frac{1}{2} \left\| \sum_{k=1}^n e_k \otimes x_k \right\|^2_{(R+C)\otimes_{min} X} \le w \left(\sum_{k=1}^n x_k \otimes \overline{x}_k \right) \le \left\| \sum_{k=1}^n e_k \otimes x_k \right\|^2_{(R+C)\otimes_{min} X} .
\end{align}
\end{lemma}
\begin{proof}
	
	We explicitly define the alluded weight. For any $x_1,\ldots , x_n \in X$,
	\begin{equation}\label{def_w}
	w\left(\sum_{k=1}^n x_k \otimes \overline{x}_k \right)  :=  \left\|  \sum_{k=1}^n e_k \otimes x_k \right\|^2_{(R +_2 C) \otimes_{min} X},
	\end{equation}
	where here $\lbrace e_k \rbrace_k$ is an orthonormal basis of $ \ell_2$.
	
	First of all, note that Equation (\ref{homog R+_2 C}) guarantees that the quantity $\left\|  \sum_{k=1}^n e_k \otimes x_k \right\|^2_{(R +_2 C) \otimes_{min} X}$ does not depend on the chosen orthonormal basis.
	
	In order to see that $w$ is well defined, consider two possible representations of an element in $(X \otimes \overline X )_+ $. That is, let $x=\sum_{i=1}^nx_i\otimes \overline{x}_i$ and $y=\sum_{j=1}^my_j\otimes \overline{y}_j$ such that $x=y$. Now, using that this implies both inequalities $x\leq y$ and $y\leq x$, one can easily deduce that $$\left\|  \sum_{i=1}^n e_i \otimes x_i \right\|_{(R +_2 C) \otimes_{min} X}=\left\|  \sum_{j=1}^m e_j \otimes y_j \right\|_{(R +_2 C) \otimes_{min} X}$$ from Proposition \ref{Prop Order} and Equation (\ref{homog R+_2 C}).

Note also that Equation (\ref{inequality weight}) is automatically satisfied by the definition of $w$ and the definition of the norm $\| x \|_{(R +C) \otimes_{min} X}$. Hence, we just need to prove that $w$ is indeed a weight. For that, we check that $w$ satisfies the conditions in Definition \ref{Def_weight}. 
	
It is clear that $\omega(x)\geq 0$ for every $x\in (X \otimes \overline{X} )_+$. In fact it is also very easy to check that $w$ is positively homogeneous. Next, let us verify the subadditivity. Consider $ \sum_{i=1}^n x_i \otimes \overline x_i, \, \sum_{j=1}^m  y_j \otimes \overline y_j \in (X \otimes \overline{X})_+$ and let us denote $x = \sum_{i=1}^n e_i \otimes x_i, \, y = \sum_{j=1}^m  e_{n+j} \otimes y_j \in \ell_2\otimes X$. Then we have that:
		\begin{align*}
		w\left( \sum_{i=1}^n x_i \otimes \overline x_i  +  \sum_{j=1}^m  y_j \otimes \overline y_j  \right) &= \big\| x    +   y   \big \|_{ (R +_2 C) \otimes_{min} X }^2  \\
		& = \inf_{T,\, S \, : \, x+y = T+S}    \| T \|_{ R \otimes_{min} X }^2   +   \| S \|_{ C \otimes_{min} X }^2  \\
		&\le \inf_{   \begin{subarray}{c} T_x,\, S_x \, : \, x = T_x + S_x ,\\ T_y,\, S_y \, : \, y = T_y + S_y \end{subarray}  }
		\| T_x + T_y \|_{R \otimes_{min} X }^2    +   \| S_x + S_y \|_{C \otimes_{min} X}^2 	\\
		& \stackrel{(*)}{ \le }    \inf_{   T_x,\, S_x \, : \, x = T_x + S_x   }
		  \| T_x \|_{R \otimes_{min} X }^2    +   \| S_x  \|_{C \otimes_{min} X}^2 
		\\ & \quad +  \inf_{   T_y,\, S_y \, : \, y = T_y + S_y  }   \| T_y \|_{R \otimes_{min} X }^2 + \| S_y \|_{C \otimes_{min} X}^2 \\
		& = 	w\left(\sum_{k=1}^n x_k \otimes \overline{x}_k \right) + 	w\left(  \sum_{j=1}^m  y_j \otimes \overline y_j  \right).	
		\end{align*}
	The inequality $(*)$ follows straightforwardly from the definition of the norms\footnote{ Using Pisier's nomenclature, cf. \cite[\S 4]{PisierOH}, these norms are 2-convex.} $ \| x \|_{ R \otimes_{min} X}$, $  \| x \|_{ C \otimes_{min} X}$.

	Finally, the monotonicity of $w$ follows easily from Proposition \ref{Prop Order} and Equation (\ref{homog R+_2 C}).
	
\end{proof}

The previous result allows us to obtain the following corollary of Theorem \ref{Thm_gammas}:
\begin{corollary}\label{cor-thm-Pisier}
Let $X$ and $Y$ be operator spaces and let $\gamma_{R\cap C}$ the semi-norm on $X\otimes Y$ defined in Theorem \ref{Thm_gammas} by the weights 
$$
w_1 \left(  \sum_i x_i \otimes \overline{x}_i \right)  =  \Big\| \sum_i e_i \otimes x_i   \Big \|_{(R\cap C) \otimes_{min} X}^2 \quad  \text{ on $(X\otimes \overline X)_+$,}
$$
 and
$$
w_2 \left( \sum_j y_j \otimes \overline{y}_j \right) =  \Big\| \sum_j e_j \otimes y_j   \Big \|_{( R +_2 C )\otimes_{min} Y}^2 \quad \text{ on $(Y\otimes \overline Y)_+$.}
$$ 

Then, if we consider the natural algebraic inclusion $X\otimes Y\iny L(X^*, Y)$ such that $z\mapsto T_z$, the following estimate holds:
\begin{align}\label{duality estimate}
\gamma_{R\cap C}(z)\leq \Gamma_{R\cap C}(T_z)\leq \sqrt{2}\, \gamma_{R\cap C}(z), 
\end{align}where $$\Gamma_{R\cap C}(T_z):=\inf\{\|a:X^*\rightarrow R\cap C\|_{cb}\|b:R\cap C\rightarrow Y\|_{cb}: T_z=b\circ a\}.$$

Moreover, given $z\in X\otimes Y$ and a constant $K$, we have that $\gamma_{R\cap C}(z) \le K$ if and only if $|\langle z, V\rangle| \leq K $ for every $V\in (X\otimes Y)^*$ such that, as a linear map $V:X\rightarrow Y^*$, $V=v_2^*\circ v_1$ for certain operators $v_1,\, v_2$ verifying $\pi_{2,w_1}(v_1:X\rightarrow H)\leq 1 $ and $\pi_{2,w_2}(v_2:Y\rightarrow H^*)\leq 1$, being $H$ an arbitrary complex Hilbert space.
\end{corollary}
\begin{proof}
First note that $w_1$ and $w_2$ are in fact weights. In the first case, this follows from \cite[Prop. 4.7]{PisierOH} and in the second case it follows from Lemma \ref{Lemma_weight_R+C}. Then, the definition of $\gamma_{R\cap C}$ associated to these weights, provided by Theorem \ref{Thm_gammas}, reads:
$$\gamma_{R\cap C}(z) = \inf_{z=\sum_i x_i \otimes y_i} \big\| \sum_i e_i \otimes x_i   \big\|_{ (R\cap C) \otimes_{min} X } \, \big\| \sum_j e_j \otimes y_j    \big\|_{ ( R +_2 C )\otimes_{min} Y }.$$

Now the estimate (\ref{duality estimate}) follows easily from Lemma  \ref{Lemma_weight_R+C} by just noticing that the norm $\Gamma_{R\cap C}$ can be equivalently written as $$\Gamma_{R\cap C}(T_z)=\inf_{z = \sum_i x_i \otimes y_i} \big\| \sum_i e_i \otimes x_i   \big \|_{ (R\cap C) \otimes_{min} X } \, \big\| \sum_j e_j \otimes y_j   \big \|_{ ( R + C )\otimes_{min} Y }.$$

The moreover part is a direct consequence of the second part of Theorem \ref{Thm_gammas}.
\end{proof}

The previous corollary allows us to prove the following key proposition.
\begin{prop}\label{prop. Mab}
Let $H$ be a complex Hilbert space, $a$ and $b$ be elements in the unit ball of  $S_2(H)$ and $M_{a,b}:B(H)\rightarrow S_1(H)$ be the linear map defined as $M_{a,b}(x)=axb$ for every $x\in B(H)$. Then, $\Gamma_{R\cap C}(M_{a,b})\leq K$ for a universal constant $K$. Moreover, $K$ can be taken $4\sqrt{2}$.
\end{prop}

For the sake of clarity, we isolate the following part of the proof as a lemma:
\begin{lemma}\label{cor 3}
	Let $A$ be a C$^*$-algebra and $H$ be a Hilbert space. Then, any $(2,R+C)$-summing map $T:A^*\rightarrow H$ verifies that $$\|T:A^*\rightarrow R_H \cap C_H\|_{cb}\leq 2\pi_{2,R+C}(T).$$
\end{lemma}
\begin{proof}
	Let us first note that Corollary \ref{key cor} allows us to obtain
	$ \pi_2(T) \le 2 \pi_{2,R+C} (T).$
	Indeed, to show that it is enough to note that  Corollary \ref{key cor} can be reinterpreted as the estimate $\| id : \tilde H \otimes_\varepsilon A^* \rightarrow R_{\tilde H} + C_{\tilde H} \otimes_{min} A^* \| \le 2$, being $\tilde H$ any Hilbert space. The desired inequality between 2-summing norms is obtained by choosing $\tilde H = \ell_2$ and recalling Equations \eqref{Equation p-summing} and \eqref{Def:pi_2,R+C}.

	With this at hand, we can invoke the extension property of 2-summing maps explained in Section 2.1. That is, the fact that  $T$ is 2-summing implies the existence of a map $\tilde T : C(K) \rightarrow H$ such that $\pi_2(\tilde T) = \pi_2 (T) \le 2 \pi_{2, R+C} (T)$ and $T = \tilde T \circ \iota $, where $\iota : A^* \rightarrow C(K) $ is the canonical embedding for a suitable compact space $K$. Lemma \ref{lemma: 2-sum} promotes $\tilde T$ to be also completely bounded when the operator space structure $R_H \cap C_H$ is considered in the image space.
	
	Finally, since $\| \iota : A^* \rightarrow C(K)\|_{cb} = \| \iota : A^* \rightarrow C(K)\| $ holds because $C(K)$ is a commutative C$^*$-algebra, we obtain the bound:
	$$
	\| T: A^* \rightarrow R_H \cap C_H \|_{cb} \le \| \iota: A^* \rightarrow C(K)\|_{cb} \| \tilde T: C(K) \rightarrow R_H \cap C_H \|_{cb} \le 2 \pi_{2,R+C}(T).
	$$
\end{proof}

\begin{proof}[Proof of Proposition \ref{prop. Mab}]
	It is well known (see \cite[Equation (12.2.5)]{EfRu91} and the comments below) that the associated tensor to $M_{a,b}$, $\widehat{M}_{a,b}$, belongs to the unit ball of $S_1(H)\hat{\otimes}S_1(H)=S_1(H\otimes H)$ (see \cite[Proposition 7.2.1]{EffrosRuanBook} for this last identification). By duality, we deduce that for any linear map $T$ such that $\|T:S_1(H)\rightarrow B(H)\|_{cb}\leq 1$, which corresponds to an element $\widehat{T}$ in the unit ball of $(S_1(H)\hat{\otimes}S_1(H))^*$, verifies $|\langle \widehat{M}_{a,b}, \widehat{T}\rangle|\leq 1$.
	
	Now, according to Corollary \ref{cor-thm-Pisier}, proving that $\gamma_{R\cap C}(\widehat{M}_{a,b})\leq 4$ implies that $\Gamma_{R\cap C}(M_{a,b})\leq 4\sqrt{2}$. The moreover part of Corollary \ref{cor-thm-Pisier} indicates that this happens if and only if $|\langle \widehat{M}_{a,b}, V\rangle|\leq 4$, for any $V\in (S_1(H)\otimes S_1(H))^*$ such that, as a linear map $V:S_1(H)\rightarrow B(H)$, it verifies that $V=v_2^*\circ v_1$, where $\pi_{2,w_1}(v_1:S_1(H)\rightarrow \tilde{H})\leq 1 $ and $\pi_{2,w_2}(v_2:S_1(H)\rightarrow \tilde{H}^*)\leq 1$, with  $\tilde{H}$ is an arbitrary complex Hilbert space. Hence, by the comments in the previous paragraph, the result will follow from the fact that such a $V$ verifies that $\|V:S_1(H)\rightarrow B(H)\|_{cb}\leq 4$. We show below that this is in fact true.
	
		On the one hand, since $\pi_{2,w_1}(v_1)=\pi_{2,RC}(v_1)\leq 1$, Corollary \ref{cor 2} implies that $$\|v_1:S_1(H)\rightarrow R_{\tilde{H}}+C_{\tilde{H}}\|_{cb}\leq 2.$$
	
On the other hand, since $\pi_{2,w_2}(v_2)=\pi_{2,R+C}(v_2)\leq 1$, this time is Lemma \ref{cor 3} which allows us to obtain
	\begin{align*}
 \| v_2 : S_1(H) \rightarrow R_{\tilde H} \cap C_{\tilde H}\|_{cb} \leq 2.
	\end{align*}
	Clearly, this can also  be read as $	\|v_2^*:R_{\tilde{H}}+C_{\tilde{H}}\rightarrow B(H)\|_{cb} \le 2$.
	
	Therefore, we conclude that $$\|V:S_1(H)\rightarrow B(H)\|_{cb}\leq \|v_1:S_1(H)\rightarrow R_{\tilde{H}}+C_{\tilde{H}}\|_{cb}\|v_2^*:R_{\tilde{H}}+C_{\tilde{H}}\rightarrow B(H)\|_{cb}\leq 4,$$which finishes the proof.
\end{proof}

With this proposition at hand, we are ready to prove our main result.
\begin{proof}[Proof of Theorem \ref{Main Thm}]
By homogeneity it suffices to show that for every linear map $T:X\rightarrow A^*$ such that $\pi_{1,cb}(T)\leq 1$, we have $\|T\|_{cb}\leq K$. To this end,  assume that $\iota:X\iny B(H)$ is a complete isometry and let us invoke the factorization theorem for $(1,cb)$-summing maps (see Remark \ref{Rem-factorization (p,cb)}) to conclude the existence of an ultrafilter $\mathcal U$ over an index set $I$,  families $(a_i)_i$, $(b_i)_i$ in the unit sphere of $S_2(H)$, a closed (operator) space $E_1\subseteq \prod S_1/\mathcal U$ and a linear map $u:E_1\rightarrow A^*$ with $\|u\|=\pi_{1,cb}(T)\leq 1$ such that the following diagram commutes:
$$\xymatrix@R=0.6cm@C=1.5cm {{\prod B(H)/\mathcal U}\ar[r]^{\mathcal M} &
	{\prod S_1/\mathcal U} \\{j(X)}\ar@{}[u]|-*[@]{\subset}\ar[r]^{\mathcal M|_{j(X)}}
	& {E_1}\ar@{}[u]|-*[@]{\subset}\ar[d]^{u}\\{X}\ar[u]^{j}\ar[r]^{T} & {A^*}}$$
Here, $j:X\iny \prod B(H)/\mathcal U$ is a complete isometry and $\mathcal M:\prod B(H)/\mathcal U\rightarrow \prod S_1/\mathcal U$ is the linear map defined by the family $(M_i)_i$, where $M_i:B(H)\rightarrow S_1(H)$ is given by $M_i(x)=a_ixb_i$ for every $i\in I$. In order to simplify notation let us denote $\hat{B}=\prod B(H)/\mathcal U$, $\hat{S}_1=\prod S_1/\mathcal U$ and $\tilde{\mathcal M}=\mathcal M|_{j(X)}$. We will show now that the previous factorization in fact implies that $T$ factorizes through $R \cap C$ with completely bounded maps. Hence, $T$ must be completely bounded.

In first place, according to Proposition \ref{prop. Mab}, $\Gamma_{R\cap C}(M_{a_i,b_i})\leq K$ for any $i \in I$. Then, one can easily deduce from Remark \ref{remark_ultraproducts} that $\Gamma_{R\cap C}(\mathcal M)\leq K$.  Moreover, this easily implies that $\Gamma_{R\cap C}(\tilde{\mathcal M})\leq K$. Indeed, if $\mathcal M=\beta\circ \alpha$, where $\|\alpha:\hat{B}\rightarrow R_{H}\cap C_{H}\|_{cb}\|\beta:R_{H}\cap C_{H}\rightarrow \hat{S}_1\|_{cb}\leq K$, where $H$ is a complex Hilbert space, we can define $\tilde{H}=\alpha(j(X))\subset H$. In virtue of the homogeneity of $R_H \cap C_H$, $\tilde{ H}$ inherits the same  $R\cap C$ operator space structure (being now the underlying Hilbert space $\tilde{ H}$ instead of $H$). Then, by denoting $\tilde{\alpha}=\alpha|_{j(X)}:j(X)\rightarrow \tilde{H}$ and $\tilde{\beta}:\beta|_{\tilde{H}}:\tilde{H}\rightarrow E_1$, it is clear that $\tilde{\mathcal M}=\tilde{\beta}\circ\tilde {\alpha}$ and $\|\tilde{\alpha}:j(X)\rightarrow R_{\tilde{H}}\cap C_{\tilde{H}}\|_{cb}\|\tilde{\beta}:R_{\tilde{H}}\cap C_{\tilde{H}}\rightarrow E_1\|_{cb}\leq K$.

Therefore, we obtain a decomposition $T=(u\circ \tilde{\beta})\circ(\tilde{\alpha}\circ j)$ such that 
\begin{align*}
\|T:X	\rightarrow A^*\|_{cb}&\leq \|\tilde{\alpha}\circ j:X\rightarrow R_{\tilde{H}}\cap C_{\tilde{H}}\|_{cb}\|u\circ \tilde{\beta}:R_{\tilde{H}}\cap C_{\tilde{H}}\rightarrow A^*\|_{cb}\\&\leq 2 \|\tilde{\alpha}:j(X)\rightarrow R_{\tilde{H}}\cap C_{\tilde{H}}\|_{cb}\|\tilde{\beta}:\tilde{H}\rightarrow E_1\|\|u:E_1\rightarrow A^*\|\\&\leq 2 \|\tilde{\alpha}:j(X)\rightarrow R_{\tilde{H}}\cap C_{\tilde{H}}\|_{cb}\|\tilde{\beta}:R_{\tilde{H}}\cap C_{\tilde{H}}\rightarrow E_1\|_{cb}\|u:E_1\rightarrow A^*\|\\&\leq 2K,
\end{align*}where in the second inequality we have used that $\|j:X\rightarrow j(X)\|_{cb}\leq 1$ and Corollary \ref{key cor} (in its dual form) to write $$\|u\circ \tilde{\beta}:R_{\tilde{H}}\cap C_{\tilde{H}}\rightarrow A^*\|_{cb}\leq 2\|u\circ \tilde{\beta}:\tilde{H}\rightarrow A^*\|\leq \|\tilde{\beta}:\tilde{H}\rightarrow E_1\|_{}\|u:E_1\rightarrow A^*\|,$$ while in the third inequality we have used the trivial inequality $\|\tilde{\beta}:\tilde{H}\rightarrow E_1\|\leq \|\tilde{\beta}:R_{\tilde{H}}\cap C_{\tilde{H}}\rightarrow E_1\|_{cb}$ together with the fact that $u$ is a contraction.

This concludes the proof.
\end{proof}
\begin{remark}
Note that we have actually proved that any linear map $T:X\rightarrow A^*$, where $X$ is an operator space and $A$ is a C$^*$-algebra, verifies $$\|T\|_{cb}\leq \Gamma_{R\cap C}(T)\leq K\pi_{1,cb}(T).$$
\end{remark}

\section{Quantum XOR games via tensor norms}\label{Sec:QXOR}

A \textit{bipartite quantum XOR game} is described by means of a family of bipartite quantum states $(\rho_x)_{x=1}^N$, a family of signs $c=(c_x)_{x=1}^N\in \{-1,1\}^N$ and a probability distribution $p=(p_x)_x$ on $\{1,\cdots, N\}$. Here, a bipartite quantum state is just a semidefinite positive operator acting on the tensor product of two finite dimensional complex Hilbert spaces, $H_A\otimes H_B$, with trace one.

In order to understand the game, we can think of two (spatially separated) players, Alice and Bob,  and a referee. The game starts with the referee choosing one of the states $\rho_x$ according to the probability distribution $p$. Then, the referee sends  register $H_A$ to Alice and register $H_B$ to Bob (this can be understood as some quantum questions). After receiving the states, Alice and Bob must answer an output, $a=\pm 1$ in the case of Alice and $b=\pm 1$ in the case of Bob. Then, the players win the game if $ab=c_x$. These games were first considered in \cite{ReVi15} as a natural generalization of classical XOR games, which have a great relevance in both quantum information and computer science. As we will see below, the relevant information of the game is encoded in the operator
\begin{align}\label{XOR quantum-operator}
G=\sum_{x=1}^Nc_xp_x \rho_x,
\end{align} which is a selfadjoint operator acting on $H_A\otimes H_B$ such that $\|G\|_{S_1(H_A\otimes H_B)}\leq 1$. In fact, one can relate any selfadjoint operator acting on $H_A\otimes H_B$ verifying $\|G\|_{S_1(H_A\otimes H_B)}\leq 1$ with a quantum XOR game, establishing a correspondence between these objects. 

In the following we will denote by $M_k$ (resp. $M_k^{\mathrm{sa}}$) the complex (real) vector space of $k\times k$ (selfadjoint) matrices. This space endowed with the trace and operator norms will be denoted by $S_1^k$ ($\Tk$) and $S_\infty^{k}$ ($S_\infty^{k,\mathrm{sa}}$), respectively. In the rest of this section we will identify $H_A=\C^n$ and $H_B=\C^m$. In this case, according to the previous paragraph, a quantum XOR game $G$ can be identified with an element in $B_{S_1^{nm,\mathrm{sa}}}$,  the unit ball of $S_1^{nm,\mathrm{sa}}$.

When playing a quantum XOR game, Alice and Bob generate their answers by means of some operation (a quantum channel, see e.g. \cite{NielsenChuang}) on the system received from the referee. We call such an operation a \textit{strategy}. Formally, a  strategy for Alice and Bob can be expressed by a linear map $\mathcal P:M_{nm}^{\mathrm{sa}} \rightarrow \R^4$ such that, for any given state $\rho$, it assigns a probability distribution over the set of possible answers: $$\mathcal P(\rho)=\mathcal P(a,b|\rho)_{a,b=\pm 1}.$$Note that, for a fixed strategy, it is very easy to write the probability of winning the game: 
\begin{align*}
\mathbf{P}_{win}(G;\mathcal P)=\sum_{x:c_x=1}p_x\Big(\mathcal P(1,1|\rho_x)+\mathcal P(-1,-1|\rho_x)\Big)+\sum_{x:c_x=-1}p_x\Big(\mathcal P(1,-1|\rho_x)+\mathcal P(-1,1|\rho_x)\Big).
\end{align*}

It is also easy to see that if Alice and Bob answer randomly (somehow the most naive strategy), that is, $\mathcal P(a,b|\rho_x)=\frac{1}{4}$ for every $a,b=\pm 1$ and every $\rho_x$, then $\mathbf{P}_{win}(G;\mathcal P)=\frac{1}{2}$. Hence, when working with XOR games, it is very common to study the so-called  \textit{bias} of the game, $\beta (G; \mathcal P)=2 (\mathbf{P}_{win}(G;\mathcal P)-1/2)$ or, equivalently,
\begin{align*}
\mathbf{P}_{win}(G;\mathcal P)-\mathbf{P}_{loose}(G;\mathcal P)=\sum_{x=1}^Np_xc_x\sum_{a,b=\pm 1} abP(a,b|\rho_x).
\end{align*}

Then, we see that, in order to compute the bias, the only relevant part of the strategies are the correlations. That is, given a strategy $\mathcal P$ and a state $\rho$, if we define $$\gamma_\mathcal P(\rho)=\sum_{a,b=\pm 1} abP(a,b|\rho),$$we have $$\beta_\mathcal P (G; \mathcal P)=\sum_{x=1}^Np_xc_x\gamma_\mathcal P(\rho_x).$$

As the reader may guess, the winning probability of the game (and so its bias) will strongly depend on the form of the strategies under consideration. The strategies considered in a given context will be determined by the \textit{resources} allowed to Alice and Bob to play the game. One extreme case is that where the players are allowed to perform any global quantum measurement. This case can be understood as if both players were located at the same place so that they can act as a single person with access to both registers $H_A$ and $H_B$. In this case, a strategy will be given by a family of semidefinite positive operators $(E_{a,b})_{a,b=-1,1}$ acting on $\C^n\otimes \C^m$ verifying that $\sum_{a,b=-1,1}E_{a,b}=\uno_{M_{nm}}$ and such that $$\mathcal P(a,b|\rho)=tr(E_{a,b} \rho) \hspace{0.2 cm} \text{for every $a,b=\pm 1$.}$$ 

It is very easy to see that the supremum of the bias of the game $G\in S_1^{nm,\mathrm{sa}}$ when the players are restricted to these kinds of strategies is given by 
\begin{align*}
\beta_{owq}(G)=\sup\{tr(X G): \, X\in B_{S_\infty^{nm,\mathrm{sa}}}\}=\|G\|_{S_1^{nm}}, 
\end{align*}where $G$ was defined in Equation (\ref{XOR quantum-operator}). The sub-index $owq$ stands for \textit{one-way quantum communication}. This is justified by the observation that the strategies considered above coincide with those where Alice is allowed to send quantum information to Bob (or the other way around) as part of their strategy, since in this case Alice can send her part of the system $\rho_x$ to Bob so that he has access to the whole state.

In this section we will be interested in the identification between the elements $G\in S_1^{n,\mathrm{sa}}\otimes S_1^{m,\mathrm{sa}}\subset S_1^n\otimes S_1^m$ and the linear maps $\hat{G}:S_\infty^n\rightarrow S_1^m$, where we recall that, given $G$, we define $\hat{G}(x)=(tr\otimes \uno_{H_B})\big(G(x^T\otimes \uno_{H_B})\big)$. Note that we must see $G$ as an element in the complex space $S_1^n\otimes S_1^m$ in order to work with operator spaces. With this identification in mind, it is well known that $$\|G\|_{S_1^{nm}}=\pi_1^o(\hat{G}:S_\infty^n\rightarrow S_1^m).$$ 

Indeed, for every linear map $\hat{G}:S_\infty^n\rightarrow S_1^m$ the completely $1$-summing norm coincides with the completely nuclear norm \cite[Corollary 15.5.4]{EfRu91} and the fact that the operator spaces are finite dimensional guarantees that the nuclear norm of $\hat{G}$ is exactly the same as $\|G\|_{S_1^{nm}}$.

Another extreme set of  strategies (somehow, at the opposite side, because they are the most limited ones) are those where Alice and Bob must answer independently. These strategies are usually called \textit{product} or \textit{unentangled strategies} \cite{ReVi15}. In this case there exist operators $E_a$ acting on $\C^n$ and $F_b$ acting on $\C^m$, for $a,b=\pm 1$ such that they are  positive semidefinite, verify $E_1+E_{-1}=\uno_{M_n}$, $F_1+F_{-1}=\uno_{M_m}$ and $$\mathcal P(a,b|\rho)=tr(E_a\otimes F_b \rho)\text{   for any $a,b=\pm 1$ and $\rho$.}$$ It is easy to see that the supremum of the bias of the game $G$ when the players are restricted to these kinds of strategies is given by 
\begin{align*}
\beta(G)=\sup\{tr(A\otimes B G): \, A\in B_{S_\infty^{n,\mathrm{sa}}}, B\in B_{S_\infty^{m,\mathrm{sa}}}\}.
\end{align*}

In particular, the ``norm expression'' of this quantity has the form $$\beta(G)= \|G\|_{S_1^{n,\mathrm{sa}}\otimes_\epsilon S_1^{m,\mathrm{sa}}}.$$One can also show \cite[Claim 4.7]{ReVi15} that
\begin{align*}
\beta(G)\leq \|G\|_{S_1^n\otimes_\epsilon S_1^m}=\|\hat{G}:S_\infty^n\rightarrow S_1^m\|\leq \sqrt{2}\beta(G).
\end{align*}

There are many more possible strategies one can consider in the study of quantum XOR games. A very important family of strategies are the so-called \textit{entangled strategies}, in which the players are allowed to use a bipartite quantum state. This situation has been deeply studied and it leads to the expression 
\begin{align*}
\beta^*(G)=\sup\{tr\big((A\otimes B)(G\otimes \rho_{A'B'}\big)\},
\end{align*}where in this case the supremum runs over all possible complex Hilbert spaces $H_{A'}$, $H_{B'}$, bipartite quantum states $\rho_{A'B'}$ acting on $H_{A'}\otimes H_{B'}$ and self adjoint contractive operators $A$ and $B$ acting on $\C^n\otimes H_{A'}$ and $\C^m\otimes H_{B'}$, respectively . In this case, the norm to be considered in $S_1^n\otimes S_1^m$ is the minimal norm (in the category of operator spaces) and one can show \cite[Claim 4.14]{ReVi15} that
\begin{align}\label{entangled_bias}
\beta^*(G)= \|G\|_{S_1^n\otimes_{min} S_1^m}= \|\hat{G}:S_\infty^n\rightarrow S_1^m\|_{cb}.
\end{align}

In light of the previous paragraphs, we see that the bias of the game $G$ according to different type of strategies can be expressed by means of different norms of $G$ as a linear map from $S_1^n$ to $S_1^m$. This is the way in which we aim to understand the bias of $G$ when the players are restricted to sending \textit{classical communication} from Alice to Bob. The study of this set of strategies is the main goal of this section.

Denoting by $\beta_{owc}(G)$ the bias of $G$ when the players are restricted to the use of one-way classical communication (from Alice to Bob), we will show:
\begin{prop}\label{Prop XOR-norms}
	Given a quantum XOR game $G\in S_1^{n,\mathrm{sa}}\otimes S_1^{m,\mathrm{sa}}\subset S_1^n\otimes S_1^m$,  we have $$\beta_{owc}(G)\leq \pi_{1,cb}(\hat{G}:S_\infty^n\rightarrow S_1^m)\leq 4 \beta_{owc}(G).$$ 
\end{prop}

In order to prove the previous proposition we must study the correlations obtained from the strategies we are considering. Let's assume that Alice can send $c$ bits of classical information (so, $2^c$ classical messages) to Bob as a part of their strategy. Hence, after receiving her part of the system from the referee, Alice will have to produce two different data: the message to be sent to Bob and the output $a$ to be sent to the referee. This can be modelled by a family of semidefinite positive operators $E_{a,k}$ acting on $\C^n$, where $a=\pm 1$, $k=1,\cdots, 2^c$,  and such that $\sum_{a,k}E_{a,k}=\uno_{M_n}$. Indeed, given a state $\rho$ acting on $\C^n$, the probability that Alice outputs the pair $(a,k)$ upon the reception of $\rho$ is given by $p(a,k)=tr(E_{a,k}\rho)$. On the other hand, after this first stage Bob will have access to his part of the state $\rho_x$ as well as the message received from Alice, and he will have to output $b=\pm 1$. Hence, Bob's action is modelled by a family of semidefinite positive operators $F_{b,k}$ acting on $\C^m$, where $a=\pm 1$, $k=1,\cdots, 2^c$, and  such that $\sum_{b}F_{b,k}=\uno_{M_m}$ for every $k$ (that is, Bob can perform a measurement according to the message received from Alice). In this way, the strategy will be given by $$\mathcal P(a,b|\rho)=\sum_{k=1}^{2^c}tr(E_{a,k}\otimes F_{b,k} \rho)\text{   for any $a,b=\pm 1$ and $\rho$}.$$ It is now very easy to see that the supremum of the bias of $G$ over all possible strategies of this form is given by 
\begin{align}\label{owc bias}
\beta_{owc} (G)=\sup\left\{\sum_{k=1}^{2^c}tr\left((A_k\otimes B_k) G\right): A_k=E_{1,k}-E_{-1,k},\, B_k=F_{1,k}-F_{-1,k}\right\},
\end{align}where here the supremum is taken over families of operators $\{E_{a,k}\}_{a,k}$ and $\{F_{b,k}\}_{b,k}$ as above.

In Proposition \ref{Prop XOR-norms} we will relate the bias $\beta_{owc} (G)$ with the $\pi_{1,cb}$-norm (defined in Equation \eqref{Equation (p,cb)-summing}) of the corresponding map $\hat G$. It easy to see that this norm can be equivalently written as
\begin{align}\label{Eq_sup_d}
\pi_{1,cb}(\hat{G}:S_\infty^n\rightarrow S_1^m)=\sup_d \left\|\uno\otimes \hat{G}:\ell_1^d\otimes_{min} S_\infty^n\rightarrow \ell_1^d(S_1^m) \right\|.
\end{align}

Let us write this norm in more detail. For each natural number $d$ in the above supremum we have:
\begin{align}
\left \|\uno\otimes \hat{G}:\ell_1^d\otimes_{min} S_\infty^n\rightarrow \ell_1^d(S_1^m) \right\| 
&=  \sup \left\{	\Big\|(\uno\otimes \hat{G})\Big(\sum_{k=1}^de_k\otimes A_k\Big)\Big \|_{\ell_1^d(S_1^m)} \, : \, \sum_{k=1}^de_k\otimes A_k\in B_{ \ell_1^d \otimes_{min} S_\infty^n } 	\right\} 
\nonumber
\\
&=  \sup \left\{ \sum_{k=1}^d  tr\big(G(A_k^T\otimes B_k)\big) 	\, : \, \begin{array}{cc} \sum_{k=1}^de_k\otimes A_k\in B_{ \ell_1^d \otimes_{min} S_\infty^n }, \\[0.3em] \hspace{-0.4 cm} \sum_{i=1}^de_k\otimes B_k\in B_{ \ell_\infty^d (S_\infty^m}) \end{array} 	\right\}.
\label{pi_1,cb:def1}
\end{align}

Note that the transpose in $A_k$ doesn't play any role in Equation (\ref{pi_1,cb:def1}) because the supremum is taken over all $x=\sum_{k=1}^de_k\otimes A_k$ with $\|x\|_{ \ell_1^d\otimes_{min} S_\infty^m} \le 1$. Replacing $A_k$ by $A_k^T$ doesn't change the norm of $x$, so we can ignore it.

As a final preamble before proving Proposition \ref{Prop XOR-norms}, we need to recall the following decomposition theorem due to Wittstock \cite{Wittstock81}. We use the statement appearing in \cite[Corollary 1.9]{Pisierbook}.
\begin{theorem}\label{Wittstock}
	Let $A$ be a C$^*$-algebra and $H$ be a Hilbert space. Then, for any completely bounded map $u:A\rightarrow B(H)$ there exist completely positive maps $u_k:A\rightarrow B(H)$ for $k=1,\ldots, 4$, such that $u=(u_1-u_2)+i(u_3-u_4)$ and \footnote{For this part of the statement, see the last line in the proof of \cite[Corollary 1.9]{Pisierbook}.} $$\max\{\|u_1+u_2\|_{cb}, \|u_3+u_4\|_{cb}\}\leq \|u\|_{cb}.$$ In particular, if $\|u\|_{cb}\leq 1$, $(u_1+u_2)(\uno_A)\leq \uno_{B(H)}$ and $(u_3+u_4)(\uno_A)\leq \uno_{B(H)}$.
\end{theorem}
\begin{proof}[Proof of Proposition \ref{Prop XOR-norms}]
	
	Let us first show that $\beta_{owc}(G)\leq \pi_{1,cb}(\hat{G}:S_\infty^n\rightarrow S_1^m)$. To do so, let us consider 
	operators $A_k=E_{1,k}-E_{-1,k}$ and $B_k=F_{1,k}-F_{-1,k}$ for every $k$ such that the $E$'s and the $F$'s are  positive semidefinite, $\sum_{a,k}E_{a,k}=\uno_{M_n}$ and $\sum_{b}F_{b,k}=\uno_{M_m}$ for every $k$. Let us show that
	\begin{align}\label{ineq_positive and cb}
	\Big\|\sum_k e_k\otimes A_k\Big\|_{\ell_1^{2^c}\otimes_{min} S_\infty^n} \leq 1\hspace{0.3 cm}\text{  and } \hspace{0.3 cm}\|B_k\|_{S_\infty^m}\leq 1\text{   for every $k$.}
	\end{align} The second bound in Equation (\ref{ineq_positive and cb}) is very easy from the definition of $B_k$ and the fact that $F_{1,k}$ and $F_{-1,k}$ are positive semidefinite verifying $F_{1,k}+F_{-1,k}=\uno_{M_m}$ for every $k$. In order to see the first bound in (\ref{ineq_positive and cb}), note that $$\Big\|\sum_k e_k\otimes A_k\Big\|_{\ell_1^{2^c}\otimes_{min} S_\infty^n}=\|\hat A : \ell_\infty^{2^c} \rightarrow S_\infty^n \|_{cb},$$where $\hat A$ is the linear map defined by $\hat A(e_k)=A_k$ for every $k$. Now, if we consider the linear maps $\hat u^{\pm} :  \ell_\infty^{2^c} \rightarrow S_\infty^n$ defined by $\hat{u}^{\pm}(e_k)=E_{\pm 1, k}$, respectively, they verify that $\hat A=\hat u^{+}-\hat u^{-}$, both maps $\hat u^{+}$  and $\hat u^{-}$ are completely positive\footnote{Since $A=\ell_\infty^{2^c}$ is a commutative C$^*$-algebra, positive maps are automatically completely positive maps.} and $\hat u^{+}+\hat u^{-}$ is a unital map. Then, using Stinespring's dilation theorem \cite[Theorem 4.1]{PaulsenBook} on the maps $\hat u^{+}$  and $\hat u^{-}$ one can check that $\hat A$ is indeed completely contractive. This proves the desired implication.

	Let us now show that $\pi_{1,cb}(\hat{G}:S_\infty^n\rightarrow S_1^m)\leq 4\beta_{owc}(G)$. According to the equations (\ref{Eq_sup_d}) and (\ref{pi_1,cb:def1}),  given $\epsilon>0$ there exist $d\in \mathbb N$, $x=\sum_{k=1}^de_i\otimes A_k$ with $\|x\|_{ \ell_1^d\otimes_{min} S_\infty^n}\leq 1$ and $\|B_k\|_{S_\infty^m}\leq 1$ for every $k$ such that  $$\pi_{1,cb}(\hat{G})\leq \sum_{k=1}^dtr(G(A_k\otimes B_k))+\epsilon.$$
	
	Next, we construct a strategy from these elements in order to bound $\beta_{owc}(G)$. On the one hand, we write $B_k=B_k^1+iB_k^2$, with $B_k^j\in S_\infty^{m,\mathrm{sa}}$, and $\|B_k^j\|\leq 1$ for $j=1,2$. On the other hand, if we realize $x$ as a completely contractive map $\hat{x}:\ell_\infty^d\rightarrow S_\infty^n$, we can apply Theorem \ref{Wittstock} to obtain completely positive maps $u_i:\ell_\infty^d\rightarrow S_\infty^n$ such that $u=(u_1-u_2)+i(u_3-u_4)$ and $$\max\{\|u_1+u_2\|_{cb}, \|u_3+u_4\|_{cb}\}\leq \|u\|_{cb}\leq 1.$$ Moreover, $(u_1+u_2)(\uno_{\ell_\infty^d})\leq \uno_{M_n}$ and $(u_3+u_4)(\uno_{\ell_\infty^d})\leq \uno_{M_n}$. Let us define $E_{1,k}=u_1(e_k)$, $E_{-1,k}=u_2(e_k)$, $\tilde{E}_{1,k}=u_3(e_k)$ and  $\tilde{E}_{-1,k}=u_4(e_k)$ for every $k=1,\cdots, d$. Note that, $\sum_{a,k}E_{a,k}\leq \uno_{M_n}$ and $\sum_{a,k}\tilde{E}_{a,k}\leq \uno_{M_n}$.  In order to sum up to one, we artificially define $E_{1,0}=\uno_{M_n}-\sum_{a,k}E_{a,k}$, $E_{-1,0}=0$, $\tilde{E}_{1,0}=\uno_{M_n}-\sum_{a,k}\tilde{E}_{a,k}$, $\tilde{E}_{-1,0}=0$. Then, if we set $C_k=E_{1,k}-E_{-1,k}$ and $\tilde{C}_k=\tilde{E}_{1,k}-\tilde{E}_{-1,k}$ for $k=0,1,\cdots, d$, we obtain a couple of families $\{C_k\}_k$ and $\{\tilde{C}_k\}_k$ as in the Equation (\ref{owc bias}). Notice that, by construction, $A_k=C_k+i\tilde{C}_k$ for $k=1,\cdots, d$.
	
	Hence, we can write
	\begin{align*}
	&\Big|\sum_{k=1}^dtr(G(A_k\otimes B_k))\Big|\leq 2\sup\left\{\Big|\sum_{k=1}^dtr(G(A_k\otimes B_k))\Big|: B_k\in B_{S_\infty^{m,\mathrm{sa}}} \right\}
	\\&\leq 2\sup\left\{\Big|\sum_{k=0}^{d}tr(G(C_k\otimes B_k))\Big|: B_k\in B_{S_\infty^{m}}\right\}+2\sup\left\{\Big|\sum_{k=0}^{d}tr(G(\tilde{C}_k\otimes B_k))\Big|: B_k\in B_{S_\infty^{m}}\right\}\\&\leq 4\beta_{owc}(G).
	\end{align*}
	
	With this, we have proved that $\pi_{1,cb}(\hat{G})\leq 4\beta_{owc}(G)+\epsilon$ for every $\epsilon>0$, from where we immediately conclude that $\pi_{1,cb}(\hat{G})\leq 4\beta_{owc}(G)$, as we wanted.
\end{proof}

Proposition \ref{Prop XOR-norms} complements the clean connection between the different values of quantum XOR games and certain norms on the corresponding linear maps associated to these games. This connection, implicitly initiated in \cite{ReVi15} (see also \cite{JuPaVi18}, where the  $1$-summing norm was used to study classical XOR games with communication), allows us to reformulate the chain of inequalities 
$$\beta(G)\leq \left\{
\begin{array}{c}
\beta^*(G)\\[0.2em]
\beta_{owc}(G)
\end{array}
\right\rbrace  \leq \beta_{owq}(G),
$$which is trivial from a physical point of view, as
$$\|\hat{G}:S_\infty^{n}\rightarrow S_1^m\|\leq \left\{
\begin{array}{c}
\|\hat{G}:S_\infty^{n}\rightarrow S_1^m\|_{cb}\\[0.2em]
\pi_{1,cb}(\hat{G}:S_\infty^{n}\rightarrow S_1^m) 
\end{array}
\right\rbrace  \leq \pi_1^o(\hat{G}:S_\infty^{n}\rightarrow S_1^m).
$$

This establishes a clear hierarchy on the relative power of different resources when playing quantum XOR games. However, this hierarchy does not say anything about the comparison between players sharing entanglement (but no communication) and players with one-way classical communication (but no entanglement). That is, the comparison between the norms  $\|\cdot\|_{cb}$ and $\pi_{1,cb}(\cdot)$.

As a first approach to understand the previous relation, we can restrict to operators acting on the diagonals of $S_\infty^{n}$ and $S_1^m$; that is, $\hat{G}:\ell_\infty^n\rightarrow \ell_1^m$ (or equivalently $G\in \ell_1^n\otimes \ell_1^m$). We have\footnote{It is well-known that for these kinds of maps $\pi_{1}(\hat{G}) =\pi_{1,cb}(\hat{G}) = \pi_{1}^o(\hat{G})$. That is, the three notions of 1-summing maps coincide.}
\begin{equation}\label{oneway_in_classicalXOR}
\pi_{1,cb}(\hat{G}:\ell_\infty^n\rightarrow \ell_1^m)=\pi_1^o(\hat{G}:\ell_\infty^n\rightarrow \ell_1^m).
\end{equation}
 Read in the context of quantum XOR games, the previous equation says that one-way classical communication allows the player to perform global measurements. 
This observation easily implies that for these kinds of maps
\begin{align}\label{classical games order}
\|\hat{G}:\ell_\infty^n\rightarrow \ell_1^m\|_{cb}\leq \pi_{1,cb}(\hat{G}:\ell_\infty^n\rightarrow \ell_1^m).
\end{align} Moreover, there exist maps for which  
\begin{align}\label{classical games gap}
\frac{\pi_{1,cb}(\hat{G}:\ell_\infty^n\rightarrow \ell_1^m)}{\|\hat{G}:\ell_\infty^n\rightarrow \ell_1^m\|_{cb}}\geq C\sqrt{\min\{n, m\}}
\end{align}for a universal constant $C$.

This last inequality is not surprising once we know that the classical Grothendieck's Theorem implies 
\begin{align}\label{Grothendieck thm cb-norm}
\|\hat{G}:\ell_\infty^n\rightarrow \ell_1^m\|_{cb}\leq K_G\|\hat{G}:\ell_\infty^n\rightarrow \ell_1^m\|.
\end{align}Hence, Equation (\ref{classical games gap}) follows from the well known estimate $\|id:\ell_1^n \otimes_\epsilon\ell_1^m\rightarrow \ell_1^{nm}\|\geq C\sqrt{\min\{n, m\}}$.

In fact, restricting to real tensors $G\in \ell_1^n\otimes \ell_1^m$ (that is, self adjoint operators)  corresponds to considering classical XOR games \cite{PaVi16}. In this sense, the previous comments are not new at all. Equation (\ref{classical games order}) means that for classical XOR games strategies using classical communication are always better than entangled strategies and, in some cases,  can actually be much better, cf. Equation \eqref{classical games gap}. Moreover, Equation (\ref{Grothendieck thm cb-norm}) tells us that for classical XOR games entangled strategies are very limited (in fact comparable to product strategies), something we already mentioned in the introduction.

One could wonder if something similar happens for general quantum XOR games or, on the contrary, in the setting of quantum XOR games one can find examples for which quantum entanglement is much more useful than classical information. Note that Equation (\ref{classical games gap}) immediately implies the existence of maps $\hat{G}:S_\infty^n\rightarrow S_1^m$ for which $\pi_{1,cb}(\hat{G})/\|\hat{G}\|_{cb}\geq C\sqrt{\min\{n, m\}}.$ However, Equation \eqref{oneway_in_classicalXOR} does not extend from $\ell_1$ to $S_1$ and, therefore, Equation \eqref{classical games order} might not hold in this more general case. In fact, such an extension of Equation \eqref{oneway_in_classicalXOR} is manifestly false. A very simple counterexample is provided by the transpose map  $ \tau: S_\infty^n \rightarrow S_1^n$, for which $\pi_{1}^o(\tau)/\pi_{1,cb}(\tau)= n.$ Indeed, it is very easy to see that $\pi_{1}^o(\tau)=n^2$ while $\pi_{1,cb}(\tau) = \| \tau \|=n$. Furthermore, the equality $\pi_{1,cb}(\tau) = \| \tau \|$ can be reinterpreted in terms of quantum XOR games as an example for which classical one-way communication does not provide any advantage at all over product strategies. Together with the result that there exist quantum XOR games for which entangled strategies attain a bias unboundedly larger than the one achieved by product strategies \cite[Theorem 1.2]{ReVi15}, this points out to the possibility that games $G$ for which $\|\hat{G}\|_{cb}/ \pi_{1,cb}(\hat{G})$ is arbitrarily large might exist. Contrary to this intuition, Theorem \ref{Main Thm} applied to $X=S_\infty^n$ and $A=S_\infty^m$ implies that this is not the case.
\begin{corollary}\label{corFinal}
	There exists a universal constant $C$ such that for every quantum XOR game $G$ $$\beta^*(G)\leq C\beta_{owc}(G).$$
\end{corollary}

\begin{proof}
	According to Equation (\ref{entangled_bias}) and Proposition  \ref{Prop XOR-norms}, $$\frac{\beta^*(G)}{\beta_{owc}(G)}\leq 4 \frac{\|\hat{G}:S_\infty^n\rightarrow S_1^m\|_{cb}}{\pi_{1,cb}(\hat{G}:S_\infty^n\rightarrow S_1^m)}\leq 4K,$$where $K$ is the constant appearing in Theorem \ref{Main Thm}. 
\end{proof}

Let us mention here that we do not known if $C$ can be taken equal to one in Corollary \ref{corFinal}. Hence, it could still
happen that quantum entanglement is strictly better than sending classical information in some instances.

To finish we make a comment about strategies that mix entanglement and one-way classical communication. From the quantum information point of view, it is well known that the access to both entanglement and one-way classical communication allows Alice to send one-way quantum communication to Bob (thanks to the quantum teleportation protocol). So we recover the value $\beta_{owq}(G)$. From the mathematical point of view, this argument can be understood by showing that the corresponding bias of the game coincides, up to a constant, with the norm $$\|\hat{G}:\ell_1\otimes_{min}S_\infty^{n}\rightarrow \ell_1(S_1^m)\|_{cb.}$$As we explained in the comments right below Theorem \ref{Main Thm}, this norm equals $\pi_1^o(\hat{G})=\| \hat G \|_{S_1^{nm}}$.

\end{document}